\documentclass{article}
\usepackage{amssymb}

\usepackage{graphicx}
\usepackage{amsmath}
\usepackage{color}
\usepackage{float}
\usepackage{array}

\newtheorem{theorem}{Theorem}[section]

\newtheorem{corollary}[theorem]{Corollary}

\newtheorem{definition}[theorem]{Definition}

\newtheorem{proposition}[theorem]{Proposition}
\newtheorem{remark}[theorem]{Remark}

\newenvironment{proof}[1][Proof]{\textbf{#1.} }{\ \rule{0.5em}{0.5em}}
\def\O{\Omega}

\def\na{\nabla}

\def\dpa{\partial}

\def\D{\Delta}

\newcommand{\cqfd}

\def\div{\mbox{div}}

\def\disp{\displaystyle}

\newtheorem{propo}{Proposition}[section]

\title{Shape stability of  a quadrature surface problem in infinite Riemannian manifolds}
\author{Ababacar Sadikhe DJITE$^{1,\,}$\footnote{ababacarsadikhe.djite@ucad.edu.sn} ,  Diaraf SECK$^{1,\,}$ \footnote{diaraf.seck@ucad.edu.sn}\\\\
$^{1}$ Laboratoire des Math\'ematiques de la D\'ecision et \\
d'Analyse Num\'erique, BP 16889, Dakar Fann, S\'enegal\\
Ecole Doctorale de Math\'ematiques et Informatique U.C.A.D. Dakar,  S\'en\'egal.}

\begin{document}
\maketitle
\begin{abstract}
In this paper, we give a simple control on how an optimal shape can be characterized. The framework of Riemannian manifold of infinite dimension is essential. And the covariant derivative plays a key role in the computation and in the analysis of qualitative properties from the shape hessian. The control depends only on the mean curvature of the domain which is a minimum or  a critical point.
\end{abstract}
{\bf Mathematical classification subject}: 49Q10, 53B20
\section{Introduction}
The search for the notion of  quadratures made a prodigious leap forward (1669-1704) thanks to Leibniz and Newton who, with the infinitesimal calculus, made the link between quadrature and derivative.
A brief remind could be interesting to see this  link with shape optimization. Regarding a bounded domain of $\Omega \subset \mathbb{R}^N$ with regular boundary, for instance $\mathcal C^2,$ $\mu$ a signed measure  compactly supported in $\Omega, $ it is well known there is a measure $\sigma$ called a balayage measure carried by the surface $\partial \Omega$ and having the same potential as $\mu$ outside $\bar \Omega,$ see for instance \cite{Kel}, \cite{Lan} for more details about this topic. And in this case, by classical approximation technique, one has the following relation:
\begin{eqnarray}\label{QS}
\int_{\partial \Omega} h d\sigma= \langle h, \mu\rangle \qquad  \forall h \in \mathcal H (\bar \Omega)
\end{eqnarray}
where $\mathcal H (\bar \Omega)$ denotes the set of functions harmonic on a neighborhood of $\bar \Omega.$ And we say that $\partial \Omega$ is a quadrature surface with respect to $\mu$ if  ($\ref{QS}$) is satisfied.\\
 This notion is closely linked with the overdetermined Cauchy elliptic  problem. And one can  claim that $\partial \Omega$ is a quadrature surface if and only if there is a solution to the following overdetermined Cauchy problem 
 \begin{equation}\label{FBP1}
\left\{
\begin{array}{ccc} 
-\Delta  u_{\Omega}&=& \mu\quad  \mbox{in}\quad \Omega\\
u_{\Omega}&=&0 \quad \mbox{on }\quad \partial\Omega  \\
-\frac{\partial u_{\Omega}}{\partial \vec\nu}&= &1 \quad \mbox{on}\quad \partial\Omega 
\end{array}
\right.
\end{equation}
The  above quadrature surface free boundary problem has some  physical motivations and   can be related  to  many areas such as  free streamlines, jets, Hele-show flows, electromagnetic shaping, gravitational problems etc. It has been intensively studied at least during the last forty years, see for example \cite{Shap},  \cite{Ef et al} and the references contained in these books for more details. Among these works, some authors have established an intimate link between the existence of quadrature surfaces and the solution of free boundary problems governed by overdetermined partial differential equations, see for instance \cite{He}, \cite{Sha1}  \cite{Sha2},  \cite{DS1} and references therin.\\
The quadrature surface problem $(\ref{FBP1})$ can be tackle by  a shape  optimization approach when $\mu$ is regular enough, for instance by taking it in $L^2 (\Omega), supp (\mu)\subset \Omega$.The authors invite the readers who get interest on  details to see for instance \cite{BLS} and \cite{DS1}.\\
Before proceeding further, let us remind that in optimisation or in the study of minimal action, one of the essential questions is the characterization of an optimum if it exists. When one is in a differentiable environment, i.e. if the objective function is differentiable as well as its constraints if any, the first derivative  and second one (hessian) play a fundamental role. In finite dimension, the characterization results are very well known even when we are in Banach spaces.\\
 On the other hand, when we have to deal with admissible sets of regular openings of $\mathbb{R}^N, N\geq 2$ containing the optimum to be characterized, the question is to be treated in a more delicate way. Indeed, if we consider a shape optimization problem where the variable is a regular open of class $\mathcal C^2$ and in which a boundary value  problem of partial differential equations is posed, there is the computation of the second derivative. Added to this, the equivalence of norms is to be handled if any exist. 
In this paper, we aim at studying these issues of characterization of critical or optimal domains in the case where the minimum of the considered shape  functional  exists,  in infinite dimensional Riemannian structures. To do so, it is crucial to find form spaces and associated metrics.\\

Finding a shape space and an associated metric is a challenging task and different approaches lead to various models. One possible approach is to  do as in   \cite{mm1} \cite{mm2}. These authors proposed, a survey of various suitable inner products is given, e.g., the curvature weighted metric and the Sobolev metric. 
There are various types of metrics on shape spaces, e.g., inner metrics \cite{BHM} \cite{mm2} like the Sobolev metrics, outer metrics \cite{BMTY}, \cite{K}, \cite{mm2}[6, 31, 43], metamorphosis metrics\cite{HTY} [27, 60] \cite{TY}, the Wasserstein or Monge-Kantorovic metric on the shape space of probability measures \cite{AGS},\cite{BB} [2, 7], the Weil-Petersson metric\cite{Ku} [35], current metrics \cite{DPTA}[16] and metrics based on elastic deformations \cite{FJSY}\cite{WR}[18, 64]. However, it is a challenging task to model both, the shape space and the associated metric. There does not exist a common shape space or shape metric suitable for all applications. Different approaches lead to diverse models. The suitability of an approach depends on the requirements in a given situation.


In recent work, it has been shown that PDE constrained shape optimization problems can be embedded in the framework of optimization on shape spaces. E.g., in \cite{SSW}, shape optimization is considered as optimization on a Riemannian shape manifold, the manifold of smooth shapes. Moreover, an inner product, which is called Steklov- Poincar\'e  metric, for the application of finite element (FE) methods is proposed in \cite{SSW1}.\\

As pointed out in \cite{Schu}, shape optimization can be viewed as optimization on Riemannian shape manifolds and the resulting optimization methods can be constructed and analyzed within this framework. This combines algorithmic ideas from \cite{abetal}  with the Riemannian geometrical point of view established in \cite{BHM}\\
In \cite{mm1} \cite{mm2}[24,25], a geometric structure of two-dimensional $\mathcal C ^{\infty}$ shapes was introduced and subsequently generalized to shapes in higher dimensions in \cite{mm3} \cite{BHM}\cite{BHM1}[5,6,26]. Essentially, closed curves (and closed higher-dimensional surfaces) are identified with mappings of the unit sphere to any shape under consideration. In two dimensions, this can be naturally motivated by the Riemannian mapping theorem. In this work, we focus on two-dimensional shapes as subsets of $\mathbb{R}^2.$ 
And  considering \cite{BLS}, \cite{DS1}, we think that it is possible  to write our work in high dimensions and even if $\Omega$ is an open set  with boundary of a compact  $N-$dimensional Riemannian noted $\mathcal M.$\\
One of our main question is the following:\\
Does it possible  to  express the Hessian of a shape  functional  to get sufficient conditions so that the critical domain of  the functional $J$ should its minimum?   
To answer  this  question, we study the positiveness  of the  quadratic form of the functional  $J$  which is  related to
the quadrature surface that is never but the following   free boundary problem
\begin{eqnarray*}
\begin{cases}
-\Delta u_{\Omega}= f \; \; \mbox{in}\, \, \Omega \\
u_{\Omega}=0 \; \; \mbox{on}\, \,\partial  \Omega \\
- \frac{\partial u_{\Omega}}{\partial \vec{ \nu}} = k \; \; \mbox{on}\, \, \partial \Omega 
\end{cases}
\end{eqnarray*} 
$k$ is a positive constant, and $f\in L^2 (\Omega), suppf \subset \Omega, \  \vec{\nu}$ is the exterior unit normal vector.
The above quadrature surface can be formulated as  the following shape  optimization  problem:
$$
\min_{\Omega \subset \mathbb{R}^{2}} J(\Omega)
$$
under the following partial differential equations contraints
\begin{eqnarray*}
\begin{cases}
-\Delta u_{\Omega}= f  \; \; \mbox{in}\, \, \Omega \\
u_{\Omega}=0 \; \; \mbox{on}\, \, \partial \Omega \\
\end{cases}
\end{eqnarray*}
where 
\begin{eqnarray}\label{ShaFunc}
J(\Omega)=-\frac{1}{2}\int_{\Omega}|\nabla u_{\Omega}|^{2}dx+\frac{k^{2}}{2}|\Omega|
\end{eqnarray}
is a real valued shape differentiable objective function, where $|\Omega|= \displaystyle \int_{\Omega} dx.$
\\ 
In \cite{BLS}, \cite{DS1}, there are all details on existence results of quadrature surface  by using  shape optimization tools. 

 And the second aim  question is  the problem  computation of the Hessian  in the  infinite Riemannian framework  and  how  it can be related to the  second shape derivative to deduce qualitative properties when the  minimum of  a regular enough shape  functional exists or when $\Omega$ is a critical point that is to say that the first derivative of $J(\Omega)$ is equal to zero.\\


The paper is organized as follows:\\
 In section $2,$  we give a brief survey, based on works in \cite{mm1}, \cite{mm2}, about  the characterization of the tangent space in a framework of Riemannian manifold of infinite dimension.\\
The section $3$ deals with the optimality condition of first order for the shape optimization and the computation of the covariant derivative. This latter  plays a key role in our final result. We shall give a direct way to compute it  which appears as a simplified expression.

In  section $4,$ we shall recall some technical but classical  computations of shape second derivative and established a result (stated as a proposition) giving the expression of the quadratic form associated to the quadrature surface problem.

The section $5$ which contains our main contributions, is devoted to the positiveness of the shape hessian in a Riemannian point of view  of infinite dimension. And, we shall propose simple control which allows to get key  information on the optimal shape domains when  these latters are strict local mminimum or critical point of the shape functional considered. 

\section{Characterization of tangent space at a point of $B_{e}$} 
The aim is to analyze the  correlation of the  Riemannian geometry on infinite dimensional maniolds $B_{e}$ with shape optimization.\\
The authors would like to stress on the fact that, what follow has been already done in pioneering works, see \cite{mm1},  \cite{mm2},  \cite{mm3}. We only reproduce some fundamental steps related to our work.\\

 Let   $\Omega$ be a simply connected and compact subset of $\mathbb{R}^{2}$ with $\Omega\neq\emptyset$ and $\mathcal{C}^{\infty}$ boundary $\partial\Omega$. As always in shape optimization, the boundary of the shape is all that metters. Thus we can identify the set of all shapes with the set of all those boundaries.\\
 
  Let $Emb(\mathbb{S}^{1},\mathbb{R}^{2})$ be the set of all  smooth embeddings on $\mathbb{S}^{1}$ in the plan $\mathbb{R}^{2}$, its elements are the injective mappings $c:\mathbb{S}^{1}\longrightarrow\mathbb{R}^{2}$. Let us $Diff(\mathbb{S}^{1})$ be the set of all $\mathcal{C}^{\infty}$ diffeomorphism on $\mathbb{S}^{1}$ wich opere diferentiably on $Emb(\mathbb{S}^{1},\mathbb{R}^{2}).$ Let us consider $B_{e}$ as the quotient of $Emb(\mathbb{S}^{1},\mathbb{R}^{2})$ under the action of $Diff(\mathbb{S}^{1})$ on $Emb(\mathbb{S}^{1},\mathbb{R}^{2})$. In terme whole we have 
$$
B_{e}(\mathbb{S}^{1},\mathbb{R}^{2}):=\{ \ [c]\ /\  c\in Emb \}\ 
\mbox{where} \  [c]:=\{c^{\prime}\in Emb\  /\ c^{\prime}\sim c \}$$
To characterize the tangent space at $B_{e}$ we start with the characterization of the tangent space at $Emb$ denoted $T_{c}Emb$ and the tangent space at the orbit of $c$ by $Diff(\mathbb{S}^{1})$ at $c$ denoted $T_{c}(Diff(\mathbb{S}^{1}).c)$. Thus the tangent space to $B_{e}$ is then identified with an additional to $T_{c}(Diff(\mathbb{S}^{1}).c)$ in $T_{c}Emb$. 
\begin{proposition}
Let $c\in Emb$, then the tangent space at $c$ to $Emb$ is given by: $T_{c}Emb=\mathcal{C}^{\infty}(\mathbb{S}^{1},\mathbb{R}^{2}).$
\end{proposition}
\begin{proof}
Let $h\in T_{c}Emb$, then $h$ is obtained by looking at a path of embeddings which passes through $c$. Let $c:I\times \mathbb{S}^{1}\longrightarrow\mathbb{R}^{2}$ be an embedding path such that $c(t,\theta)=c(\theta)+th(\theta)$ where $h:\mathbb{S}^{1}\longrightarrow\mathbb{R}^{2}$ we have : $\frac{d}{dt}_{|t=0}c(t,\theta)=h(\theta)$. Since $c(t,\theta)$ is an embedding patht then $c(t,\theta)$ is an immersion thus  $$T_{c}Emb=Im(T_{0}c(t,\theta))=\mathcal{C}^{\infty}(\mathbb{S}^{1},\mathbb{R}^{2}).$$ 
\end{proof}
\begin{proposition}\label{proposition-2}
The tangent space to the orbit of $c$ by $Diff(\mathbb{S}^{1})$ is the subspace of  $T_{c}Emb$ formed by vectors $m(\theta)$ of type $c_{\theta}(\theta)=c^{'}(\theta)$ times a function.
\end{proposition}
\begin{proof}
We have $Diff(\mathbb{S}^{1}).c \subset Emb$ because these are all the bijective reparametrizations of the same curve $c(\theta)$ therefore $T_{c}(Diff(\mathbb{S}^{1}).c) \subset T_{c}Emb$. Let $m\in T_{c}(Diff(\mathbb{S}^{1}).c)$ then $m$ is obtained by looking at a family of parametrizations $c(t,\theta):=c(\phi(t,\theta))$ of the curve $c(\theta)$ where $$\phi(t,.):\mathbb{S}^{1}\longrightarrow \mathbb{S}^{1}$$ is a diffeomorphism of \  $\mathbb{S}^{1}$ where $s\in \mathbb{S}^{1}$ and $t$ is the parameter of the variation of the reparametrization $\phi(t,s)$ of \  $\mathbb{S}^{1}$. We have  $\frac{d}{dt}_{|t=0}c(t,\theta)=c^{\prime}(\theta)\frac{d}{dt}_{|t=0}\phi(0,\theta)$ since $c(t,\theta)$ is a parametrization of the curve $c(\theta)$ so it an immersion then we have $$T_{c}(Diff(\mathbb{S}^{1}).c)=Im(T_{0}c(t,\theta))=c^{\prime}(\theta)\frac{d}{dt}_{|t=0}\phi(0,\theta).$$
\end{proof}
\begin{remark}
The choice of the supplementary must be coherent with the action  of $Diff(\mathbb{S}^{1})$ i.e we choose a supplementary at  $T_{c}(Diff(\mathbb{S}^{1}).c)$ in $T_{c}Emb$ invariably by the action of $Diff(\mathbb{S}^{1})$. For that it suffices to define a metric on $Emb$ for which $Diff(\mathbb{S}^{1})$ acts isometrically and to define the supplementary of $T_{c}(Diff(\mathbb{S}^{1}).c)$ as its orthogonal for this metric.
\end{remark}
\begin{definition}
Let us $G^{0}$ be an invariant metric by the action of $Diff(\mathbb{S}^{1})$ on the manifold $Emb(\mathbb{S}^{1},\mathbb{R}^2)$, defined by the application:
$$
\begin{array}{ccccl}
G^{0} & : & T_{c}Emb\times T_{c}Emb & \to & \mathbb{R} \\
 & & (h,m) & \mapsto & \displaystyle \int_{\mathbb{S}^{1}}\big<h(\theta),m(\theta)\big>|c^{\prime}(\theta)|d\theta \\
\end{array}
$$
where $\big<h(\theta),m(\theta)\big>$ is the scalare product of $h(\theta)$ and $m(\theta)$ in $\mathbb{R}^{2}$. 
\end{definition}
 \begin{proposition}\label{proposition-4}
 Let $c\in B_{e}$ then $T_{c}B_{e}$ is colinear with the vector fields following the unit normal outside the form $\Omega$. In other words $$T_{c}B_{e}\simeq\{ h\ |\  h=\alpha \vec{\nu},\alpha\in\mathcal{C}^{\infty}(\mathbb{S}^{1},\mathbb{R}) \}.$$
 \end{proposition}
\begin{proof} 
From the results shown above  the orthogonal of $T_{c}(Diff(\mathbb{S}^{1}).c)$ in $T_{c}Emb$ is the set of $h(\theta)$ in $T_{c}Emb$ which are orthogonal for the metric $G^{0}$ to all $m(\theta)=\frac{d}{dt}_{|t=0}\phi(0,\theta)c^{\prime}(\theta)$ this means that $h(\theta)$ must be perpendicular to $c^{\prime}(\theta)$ so $h(\theta)=\alpha(\theta)\vec\nu(\theta)$ where $\alpha(\theta)\in\mathcal{C}^{\infty}(\mathbb{S}^{1},\mathbb{R}).$ Therefore we have 
 $$
 T_{c}B_{e}\simeq\{ h | h=\alpha \vec{\nu},\ \alpha\in\mathcal{C}^{\infty}(\mathbb{S}^{1},\mathbb{R}) \}
 $$
where $\vec{\nu}$ is the unit normal outside the form $\Omega$ defined at the boundary by $\partial\Omega=c$ such that $\vec{\nu}(\theta)\perp c^{\prime}(\theta)$ for all $\theta\in S^{1}$ and $c^{\prime}$ defined the circumferential derivative.
\end{proof}
Now let us consider the following therminology:
$$ds=|c_{\theta}|d\theta\qquad\mbox{arc length}. $$
\begin{definition}
A Sobolev-type metric on the manifold $B_{e}(\mathbb{S}^{1},\mathbb{R}^{2})$ is map:
$$
\begin{array}{ccccl}
G^{A} & : & T_{c}B_{e}\times T_{c}B_{e} & \to & \mathbb{R} \\
 & & (h,m) & \mapsto & \displaystyle \int_{\mathbb{S}^{1}}(1+AK_{c}^{2}(\theta))\big<h(\theta),m(\theta)\big>|c^{\prime}(\theta)|d\theta \\
\end{array}
$$
where $K_{c}$ is the sectional curvature of $c$ and $A$ a positive real. 
\end{definition}
\begin{remark}
\begin{enumerate}
\item By setting $h=\alpha\vec{\nu}$, $m=\beta\vec{\nu}$ and by parametrizing $c(s)$ by arc length we have :
$$
G^{A}(h,m)=\int_{\partial\Omega}(1+AK_{c}^{2}(\theta))\alpha\beta ds.
$$
\item If $A> 0$ $G^{A}$ is a Riemannian metric.
\end{enumerate}
\end{remark}
\section{Optimality condition of first order and covariant derivative}
The shape optimization problem we have, consists in finding the solution of the following optimization problem:
$$\min_{\Omega} J(\Omega)$$ $J(\Omega)=-\frac{1}{2}\displaystyle \int_{\Omega}|\nabla u_{\Omega}|^{2}dx+\frac{k^{2}}{2}|\Omega|$ is a  shape functional. We  seek the shape derivative associated with the functional $J(\Omega)$ following the direction of the vector field $V:\mathbb{R}^{2}\to\mathbb{R}^{2},$  $\mathcal{C}^{\infty}$ class :
$$
dJ(\Omega)[V]=\displaystyle \int_{\partial\Omega}\left(k^{2}-\left(\frac{\partial u_{\Omega}}{\partial\vec\nu}\right)^{2}\right)\big<V,\vec\nu\big>d\sigma.
$$
If $V_{|\partial\Omega}=\alpha\vec{\nu}$ we can still write :
\begin{eqnarray}\label{GRD1}
dJ(\Omega)[V]=\int_{\partial\Omega}\left(k^{2}-\left(\frac{\partial u_{\Omega}}{\partial\vec\nu}\right)^{2}\right)\alpha d\sigma.
\end{eqnarray}
It should be noted that there is a link between the shape derivative of $J$ and the gradient in Riemannian structures see \cite{Schu} and \cite{W}. To illustrate our claiming, let us consider  the Sobolev metric $G^A$ to ease the understanding of the computations. But the  authors think that it is quite possible  to generalize this  study in higher dimension than two and even with other metrics.\\
Our purpose is to  calculate the gradient of $J: B_{e}\to\mathbb{R}$ then we have :
\begin{equation}\label{GRD2}
dJ(\Omega)[V]=G^{A}(grad J(\Omega), V)
\end{equation}
if $V_{|\partial\Omega}=h$ we have
\begin{eqnarray}
dJ_{c}(h)&=&G^{A}(grad J(\Omega), h)\nonumber\\
dJ_{c}(h)&=&\int_{\partial\Omega}\left(1+AK^{2}_{c}\right)gradJ\alpha. \nonumber
\end{eqnarray}
But from (\ref{GRD2}), $$dJ_{c}(h)=\int_{\partial\Omega}\left(k^{2}-\left(\frac{\partial u_{\Omega}}{\partial\vec\nu}\right)^{2}\right)\alpha d\sigma$$ and thus 
$$
\int_{\partial\Omega}\left(k^{2}-\left(\frac{\partial u_{\Omega}}{\partial\vec\nu}\right)^{2}\right)\alpha d\sigma=\int_{\partial\Omega}\left(1+AK^{2}_{c}\right)gradJ\alpha d\sigma
$$
so that $$gradJ=\frac{1}{1+AK^{2}_{c}}\left(k^{2}-\left(\frac{\partial u_{\Omega}}{\partial\vec\nu}\right)^{2}\right).$$
The next step is  to compute the explicit form of the covariant derivative $\nabla_{h}m\in T_{c}B_{e}$ with $h, m\in T_{c}B_{e}.$\\
 The following result has been already established in a pioneering work, see \cite{Schu}. We only bring another way in the proof and additional details in the computations of the covariant derivative.
 In the last part of the paper where we think it contains our main contributions, the covariant derivative plays a key role  in the study of the positiveness of the quadratic form. We shall come back to this fact.
\begin{theorem}
Let $\Omega\subset\mathbb{R}^{2}$ at least of class $\mathcal C ^2,$ $V, W\in\mathcal{C}^{\infty}(\mathbb{R}^{2},\mathbb{R}^{2})$ vector fields which are orthogonal to the boundaries i.e  $$V_{|\partial\Omega}=\alpha\vec{\nu}$$ with $\alpha:=\big<V_{|\partial\Omega},\vec{\nu}\big>$ and $$W_{|\partial\Omega}=\beta\vec{\nu}$$ with $\beta:=\big<W_{|\partial\Omega},\vec{\nu}\big>$ such that $V_{|\partial\Omega}=h:=\alpha\vec{\nu}$,\ $W_{|\partial\Omega}=m=:\beta\vec{\nu}$ belongs to the tangent space of $B_{e}$. Then the covariant derivative associate with the Riemannian metric $G^{A}$ can be expressed as follows :
\begin{eqnarray}
\nabla_V{W}:&=&\nabla_{h}{m}=\frac{\partial\beta}{\partial\vec{\nu}}\alpha+\left(\frac{3AK_{c}^{3}+K_{c}}{1+AK_{c}^{2}}\right)\alpha\beta\nonumber\\
 &=&\big<D_{V}W,\vec{\nu}\big>+\left(\frac{3AK_{c}^{3}+K_{c}}{1+AK_{c}^{2}}\right)\big<V,\vec{\nu}\big>\big<W,\vec{\nu}\big>\nonumber.
\end{eqnarray}
where $D_{V}W$ is the directional derivative of the vector field $W$ in the direction $V$.
\end{theorem}
\begin{proof}
To calculate $\nabla_{V}W$ we use the compatibility of the  metric with the connection $\nabla$.
We have $hG^{A}(m,l)=G^{A}(\nabla_{h}m,l)+G^{A}(m,\nabla_{h}l)$ for all $l\in T_{C}B_{e}$.
Let $Z$ be a vector fields such that $l:=Z_{|\partial\Omega}=\gamma\vec{\nu}$,
we set $$F(c_{t}(\theta))=(1+AK_{c}^{2}(\theta))\big<m(\theta),l(\theta)\big>$$ thus $G^{A}(m,l)=\displaystyle \int_{\mathbb{S}^{1}}F(c_{t}(\theta))|c_{t}^{\prime}(\theta)|d\theta$ then we calculate the following expression 
$$
h(G^{A}(m,l))=\frac{d}{dt}_{|t=0}
\begin{pmatrix}
\displaystyle \int_{\mathbb{S}^{1}}F(c_{t}(\theta))|c_{t}^{\prime}(\theta)|d\theta\end{pmatrix}[V]
$$
where $c_{t}(\theta)$ denotes a family of (parameterized) curves with $c_{0}(\theta)=c(\theta)$ and $c_{t}^{\prime}(\theta)$  denotes the derivative with respect to $\theta$ of the curve $c_{t}:\theta\longrightarrow c_{t}(\theta)$. We have
\begin{eqnarray}
h(G^{A}(m,l))&=&\int_{\mathbb{S}^{1}}\begin{pmatrix}\frac{\partial[(1+AK_{c}^{2})\beta\gamma]}{\partial\vec{\nu}}\alpha|c_{t}^{\prime}(\theta)|+\frac{\partial(|c_{t}^{\prime}(\theta)|)}{\partial\vec{\nu}}(1+AK_{c}^{2})\beta\gamma\alpha)\end{pmatrix}d\theta\nonumber\\
&=&\int_{\mathbb{S}^{1}}\begin{pmatrix}
 2AK_{c}(\frac{\partial K_{c}}{\partial\vec{\nu}})\alpha\beta\gamma+(1+AK_{c}^{2})\frac{\partial\beta}{\partial\vec{\nu}}\gamma\alpha+(1+AK_{c}^{2})\frac{\partial\gamma}{\partial\vec{\nu}}\beta\alpha\end{pmatrix}d\theta \nonumber\\
& &+\int_{\mathbb{S}^{1}}\frac{\partial|c_{t}^{\prime}(\theta)|}{\partial\vec{\nu}}(1+AK_{c}^{2})\beta\gamma\alpha d\theta\nonumber
\end{eqnarray}
Now let  us calculate $\frac{\partial K_{c}}{\partial\vec{\nu}}.$ We have :
\begin{eqnarray*} \frac{\partial K_{c}}{\partial\vec{\nu}}=\frac{\big<\vec\nu,c_{\theta}\big>}{|c_{\theta}|^{2}}K_{\theta}+\frac{\big<\vec\nu,ic_{\theta}\big>}{|c_{\theta}|}K^{2}+\frac{1}{|c_{\theta}|}\left(\frac{1}{|c_{\theta}|}\left(\frac{\big<\vec\nu,ic_{\theta}\big>}{|c_{\theta}|}\right)_{\theta}\right)_{\theta}.
\end{eqnarray*}
Then we have $\big<\vec\nu, c_{\theta}\big>=0$ because $\vec\nu\perp c_{\theta}$  and  moreover,
\begin{eqnarray}
\frac{\big<\vec\nu,ic_{\theta}\big>}{|c_{\theta}|}&=&\big<\vec\nu, \frac{ic_{\theta}}{|c_{\theta}|}\big>\nonumber\\
\frac{\big<\vec\nu,ic_{\theta}\big>}{|c_{\theta}|}&=&\big<\vec\nu, \vec\nu \big>\nonumber\\
\frac{\big<\vec\nu,ic_{\theta}\big>}{|c_{\theta}|}&=&\|\vec\nu\|^{2}=1\nonumber
\end{eqnarray}
Hence, we obtain that:
\begin{eqnarray*}
 \frac{\partial K_{c}}{\partial\vec{\nu}}=K_{c}^{2}+\frac{1}{|c_{\theta}|}\left(\frac{1}{|c_{\theta}|}\left(\frac{\big<\vec\nu,ic_{\theta}\big>}{|c_{\theta}|}\right)_{\theta}\right)_{\theta}. 
 \end{eqnarray*}
  Let us compute step by step the above last term in the right hand side.\\
 First,   we have 
\begin{eqnarray*}
\left(\frac{\big<\vec\nu,ic_{\theta}\big>}{|c_{\theta}|}\right)_{\theta}&=&\frac{\partial}{\partial\theta}\left(\frac{\big<\vec\nu,ic_{\theta}\big>}{|c_{\theta}|}\right)\nonumber\\
&=&\frac{\frac{\partial}{\partial\theta}\big<\vec\nu,ic_{\theta}\big>|c_{\theta}|-\frac{\partial|c_{\theta}|}{\partial\theta}\big<\vec\nu,ic_{\theta}\big>}{|c_{\theta}|^{2}}\nonumber\\
&=&\frac{\frac{\partial}{\partial\theta}\big<\vec\nu,ic_{\theta}\big>|c_{\theta}|}{|c_{\theta}|^{2}}\nonumber\\
&=&\frac{\frac{\partial}{\partial\theta}\big<\vec\nu,ic_{\theta}\big>}{|c_{\theta}|}
\end{eqnarray*}
\begin{eqnarray*}
\left(\frac{\big<\vec\nu,ic_{\theta}\big>}{|c_{\theta}|}\right)_{\theta}&=&\frac{\big<\frac{\partial\vec\nu}{\partial\theta},ic_{\theta}\big>+\big<\vec\nu,i\frac{\partial c_{\theta}}{\partial \theta}\big>}{|c_{\theta}|}\nonumber\\
&=&\frac{\big<\vec\nu^{\prime}(\theta),ic_{\theta}\big>+\big<\vec\nu(\theta),ic_{\theta\theta}\big>}{|c_{\theta}|}\nonumber\\
&=&\frac{\big<\vec\nu^{\prime}(\theta),ic_{\theta}\big>}{|c_{\theta}|}+\frac{\big<\vec\nu(\theta),ic_{\theta\theta}\big>}{|c_{\theta}|}\nonumber\\
&=&\big<\vec\nu^{\prime}(\theta),\frac{ic_{\theta}}{|c_{\theta}|}\big>+\frac{\big<\vec\nu(\theta),ic_{\theta\theta}\big>}{|c_{\theta}|}\nonumber\\
&=&\big<\vec\nu^{\prime}(\theta),\vec\nu(\theta)\big>+\frac{\big<\vec\nu(\theta),ic_{\theta\theta}\big>}{|c_{\theta}|}\nonumber
\end{eqnarray*}
Note that   $\|\vec\nu\|^{2}=1$ which is never but $\big<\vec\nu,\vec\nu\big>=1.$ Therefore, by differentiation, we have:
\begin{eqnarray}
\big<\vec{\nu}^{\prime}(\theta),\vec\nu(\theta)\big>+\big<\vec\nu(\theta),\vec\nu^{\prime}(\theta)\big>&=&0\nonumber\\
2\big<\vec\nu^{\prime}(\theta),\vec\nu(\theta)\big>&=&0\nonumber\\
\big<\vec\nu^{\prime}(\theta),\vec\nu(\theta)\big>&=&0.\nonumber
\end{eqnarray}
Indeed, proceeding further the computation, we have:
\begin{eqnarray}
\left(\frac{\big<\vec\nu,ic_{\theta}\big>}{|c_{\theta}|}\right)_{\theta}&=&\frac{\big<\vec\nu(\theta),ic_{\theta\theta}\big>}{|c_{\theta}|}\nonumber\\
&=&\big<\vec\nu(\theta),\frac{ic_{\theta\theta}}{|c_{\theta}|}\big>\nonumber\\
&=&\big<\vec\nu(\theta),\frac{-K_{c}|c_{\theta}|c_{\theta}}{|c_{\theta}|}\big>\nonumber\\
&=&\big<\vec\nu(\theta),-K_{c}c_{\theta}\big>\nonumber\\
&=&-K_{c}\big<\vec\nu(\theta),c_{\theta}\big>=0\nonumber
\end{eqnarray}
Finally, from all the above steps, we have:
$ \frac{1}{|c_{\theta}|}\left(\frac{1}{|c_{\theta}|}\left(\frac{\big<\vec\nu,ic_{\theta}\big>}{|c_{\theta}|}\right)_{\theta}\right)_{\theta}=0$  and  we get  $$\frac{\partial K_{c}}{\partial\vec{\nu}}=K^{2}_{c}.$$
Therefore we have
\begin{eqnarray}
h(G^{A}(m,l))&=&\int_{\partial\Omega}\begin{pmatrix}
 2AK_{c}(\frac{\partial K_{c}}{\partial\vec{\nu}})\alpha\beta\gamma+(1+AK_{c}^{2})\frac{\partial\beta}{\partial\vec{\nu}}\gamma\alpha+(1+AK_{c}^{2})\frac{\partial\gamma}{\partial\vec{\nu}}\beta\alpha\end{pmatrix}d\theta \nonumber\\
& &+\int_{\partial\Omega}\frac{\partial|c_{t}^{\prime}(\theta)|}{\partial\vec{\nu}}(1+AK_{c}^{2})\beta\gamma\alpha d\theta\nonumber\\
&=&\int_{\partial\Omega}\begin{pmatrix}
 2AK_{c}\times K^{2}_{c}\alpha\beta\gamma+(1+AK_{c}^{2})\frac{\partial\beta}{\partial\vec{\nu}}\gamma\alpha+(1+AK_{c}^{2})\frac{\partial\gamma}{\partial\vec{\nu}}\beta\alpha\end{pmatrix}d\theta \nonumber\\
& &+\int_{\partial\Omega}\frac{\partial|c_{t}^{\prime}(\theta)|}{\partial\vec{\nu}}(1+AK_{c}^{2})\beta\gamma\alpha d\theta
\end{eqnarray}
Let  us calculate now the following expression 
$$
\frac{\partial(|c_{t}^{\prime}(\theta)|)}{\partial\vec{\nu}}
$$
To do this we parametrized $c(\theta)$ by arc length i.e $|c^{\prime}(\theta)|=1$. Since $$\big<c^{\prime}(\theta), c^{\prime}(\theta)\big>=1$$ and differentiating it, we have: $$\big<c^{\prime\prime}(\theta), c^{\prime}(\theta)\big>=0.$$ Then $c^{\prime\prime}(\theta)=c_{\theta\theta}(\theta)$ is proportional to $\vec\nu(s)$ so $c^{\prime\prime}(\theta)=K_{c}(\theta)\vec\nu(\theta)$ (this is the  definition  of the curvature of the curve $c$).\\
Let us compute now $\frac{d}{dt}(|c_{t}^{\prime}(\theta)|)$ at $t=0$ where 
\begin{eqnarray}
|c_{t}^{\prime}(\theta)|&=&|\frac{d}{d\theta}(c(\theta)+t\vec\nu(\theta))|\nonumber\\
              &=&|c^{\prime}(\theta)+t\vec\nu^{\prime}(\theta)|\nonumber\\
              &=&(|c^{\prime}(\theta)|^{2}+t^{2}|\vec\nu^{\prime}(\theta)|^{2}+2t\big<c^{\prime}(\theta),\vec\nu^{\prime}(\theta)\big>)^{\frac{1}{2}}
\end{eqnarray}
From a Taylor's expansion of the previous expression  in  $t$, we see that :
$$
\frac{d}{dt}_{|t=0}|c_{t}^{\prime}(\theta)|=\big<c^{\prime}(\theta),\vec\nu^{\prime}(\theta)\big>
$$
and since $$\big<c^{\prime}(\theta),\vec\nu(\theta)\big>=0$$ by differentiating we have
 $$
 \big<c^{\prime}(\theta),\vec\nu^{\prime}(\theta)\big>=-\big<c^{\prime\prime}(\theta),\vec\nu(\theta)\big>=K_{c},
 $$ 
 and  hence $\frac{d}{dt}(|c_{t}^{\prime}(\theta)|)=K_{c}$.\\
One can conclude that:
\begin{eqnarray}
h(G^{A}(k,l))&=&\int_{\partial\Omega}\begin{pmatrix} 2AK_{c}^{3}\alpha\beta\gamma + (1+AK_{c}^{2})\frac{\partial\beta}{\partial\vec{\nu}}\gamma\alpha\end{pmatrix}ds \nonumber\\
& +&\int_{\partial\Omega}\begin{pmatrix}(1+AK_{c}^{2})\frac{\partial\gamma}{\partial\vec{\nu}}\beta\alpha+K_{c}(1+AK_{c}^{2})\alpha\beta\gamma\end{pmatrix}ds.
\end{eqnarray}
We have \begin{eqnarray}
G^{A}(\nabla_{h}m,l)&=&\int_{\partial\Omega}(1+AK^{2}_{c})\nabla_{h}m\gamma\nonumber\\
&=&\int_{\partial\Omega}(1+AK^{2}_{c})\nabla_{V}W\gamma\nonumber
\end{eqnarray} 
and
\begin{eqnarray}
G^{A}(m,\nabla_{h}l)&=&\int_{\partial\Omega}(1+AK^{2}_{c})\beta\nabla_{h}l\nonumber\\
&=&\int_{\partial\Omega}(1+AK^{2}_{c})\beta\nabla_{V}Z.\nonumber
\end{eqnarray} 
Therefore
\begin{eqnarray}
G^{A}(\nabla_{h}m,l)+G^{A}(m,\nabla_{h}l)&=&\int_{\partial\Omega}(1+AK^{2}_{c})\nabla_{V}W\gamma+\int_{\partial\Omega}(1+AK^{2}_{c})\beta\nabla_{V}Z\nonumber\\
&=&\int_{\partial\Omega}(1+AK^{2}_{c})\left(\nabla_{V}W\gamma+\beta\nabla_{V}Z\right)ds.\nonumber
\end{eqnarray}
and
\begin{eqnarray}
\int_{\partial\Omega}(1+AK^{2}_{c})\left(\nabla_{V}W\gamma+\beta\nabla_{V}Z\right)ds&=&\int_{\partial\Omega}\begin{pmatrix} 2AK_{c}^{3}\alpha\beta\gamma+(1+AK_{c}^{2})\frac{\partial\beta}{\partial\vec{\nu}}\gamma\alpha\end{pmatrix}ds \nonumber\\
& +&\int_{\partial\Omega}\begin{pmatrix}(1+AK_{c}^{2})\frac{\partial\gamma}{\partial\vec{\nu}}\beta\alpha+K_{c}\alpha\beta\gamma+AK_{c}^{3}\alpha\beta\gamma\end{pmatrix}ds.\nonumber\\
&=&\int_{\partial\Omega} \left(3AK_{c}^{3}+K_{c}\right)\alpha\beta\gamma+(1+AK_{c}^{2})\frac{\partial\beta}{\partial\vec{\nu}}\gamma\alpha\nonumber\\
&+&(1+AK^{2}_{c})\frac{\partial\gamma}{\partial\vec{\nu}}\beta\alpha ds \nonumber
\end{eqnarray}
\begin{eqnarray}
\nabla_{V}W\gamma+\beta\nabla_{V}Z&=&\left(\frac{3AK^{3}_{c}+K_{c}}{1+AK^{2}_{c}}\right)\alpha\beta\gamma+\frac{\partial\beta}{\partial\nu}\gamma\alpha+\frac{\partial\gamma}{\partial\nu}\beta\alpha\nonumber\\
\nabla_{V}W\gamma&=&\left(\frac{3AK^{3}_{c}+K_{c}}{1+AK^{2}_{c}}\right)\alpha\beta\gamma+\frac{\partial\beta}{\partial\vec\nu}\gamma\alpha+\frac{\partial\gamma}{\partial\vec\nu}\beta\alpha-\beta\nabla_{V}Z\nonumber
\end{eqnarray}
By pointing  out that  $$\nabla_{V}Z=\frac{\partial\gamma}{\partial\vec\nu}\alpha,$$
we have:
\begin{eqnarray}
\nabla_{V}W\gamma&=&\left(\frac{3AK^{3}_{c}+K_{c}}{1+AK^{2}_{c}}\right)\alpha\beta\gamma+\frac{\partial\beta}{\partial\vec\nu}\gamma\alpha+\beta\frac{\partial\gamma}{\partial\vec\nu}\alpha-\beta\frac{\partial\gamma}{\partial\vec\nu}\alpha\nonumber\\
&=&\left(\frac{3AK^{3}_{c}+K_{c}}{1+AK^{2}_{c}}\right)\alpha\beta\gamma+\frac{\partial\beta}{\partial\vec\nu}\gamma\alpha\nonumber\\
\end{eqnarray}
Finally, we have:
\begin{eqnarray}
\nabla_{V}W&=&\left(\frac{3AK^{3}_{c}+K_{c}}{1+AK^{2}_{c}}\right)\alpha\beta+\frac{\partial\beta}{\partial\vec\nu}\alpha\nonumber.
\end{eqnarray}
\end{proof}
\begin{remark}
It is quite possible to begin the proof of the above theorem  by the 
 application of shape calculus rules for volume and boundary functionals as in \cite{DZ}, \cite{sozo}, \cite{HP} on the following functional
 \begin{eqnarray*}
 \int_{\partial \Omega} (1+ A K_c^2) \alpha \beta d\sigma
 \end{eqnarray*}
 The remaining computations are almost similar. We only underline  that at the end it is necessary, to see that the local covariant derivative $\nabla_{X_x} Y= \frac{d}{dt_{\vert t=0}}(Y(x+t X(x)))$ where $Y= X= \vec\nu$ and $D\vec\nu\,  \vec\nu= 0$ since $|\vec\nu|^2= 1, D\vec\nu$ being the Jacobian matix.
\end{remark}
\begin{remark}
 Let us now  calculate the torsion of the connection $\nabla.$
Indeed, one is wondering if  the torsion of the connection $\nabla$  coincides with the Levi-Civita connection.\\
 We have 
\begin{eqnarray}
T(V,W)&=&\nabla_{V}W-\nabla_{W}V-[V,W]\nonumber\\
T(V,W)&=&\big<D_{V}W,\vec{\nu}\big>+\left(\frac{3AK^{3}_{c}+K_{c}}{1+AK_{c}^{2}}\right)\big<V,\vec{\nu}\big>\big<W,\vec{\nu}\big>\nonumber\\
          &-&\big<D_{W}V,\vec{\nu}\big>-\left(\frac{3AK^{3}_{c}+K_{c}}{1+AK_{c}^{2}}\right)\big<V,\vec{\nu}\big>\big<W,\vec{\nu}\big>-[V,W]\nonumber\\
T(V,W)&=&\frac{\partial\beta}{\partial\vec{\nu}}\alpha-\frac{\partial\alpha}{\partial\vec{\nu}}\beta-[h,m].\nonumber
\end{eqnarray}
But $$\frac{\partial\beta}{\partial\vec{\nu}}\alpha-\frac{\partial\alpha}{\partial\vec{\nu}}\beta=[h,m].$$
Then we have:
\begin{eqnarray}
T(V,W)&=&[h,m]-[h,m]\nonumber\\
T(V,W)&=&0.\nonumber
\end{eqnarray}
As conslusion, we claim that
 $\nabla$ is compatible with the metric $G^{A}$ and its torsion is zero,so it coincides with the Levi-Civita connection. 
\end{remark}

\section{ Sufficient condition for the minimality of a  shape functional}

 In this section,  assuming  at first that there are at least one critical point, we shall  first present the sufficient condition on the existence of a local minimum  for a functional \ $J(\Omega)$  given as follows:

 \begin{equation}\label{l}
    J(\Omega) = \displaystyle\int_\Omega \ f_0(u_\Omega, \nabla u_\Omega)
 \end{equation}
 where $f_0$ is a function of $\mathbb{R} \times \mathbb{R}^n$ that we suppose to be smooth and $u_\Omega$ denotes a smooth solution of a boundary value problem.\\
 And in the second part,  in the case where  $J(\Omega)=-\frac{1}{2}\displaystyle \int_{\Omega}|\nabla u_{\Omega}|^{2}dx+\frac{k^{2}}{2}|\Omega|, $ we compute the second shape derivative.
 
 The fundamental question is then to study  the existence of  the local strict  minima of this functional under possible constraints that $\Omega$ is unknown but is a critical point. That means the first  derivative with respect to the domain is equal to zero at the domain $\Omega.$ We shall examine, for that, how varies this solution $u_{\Omega}$ when its domain of definition $\Omega$ moves.\\
 
 Let us recall the classic method of studying a critical point. Let $(B, \ \| \ . \ \|_1)$ be a Banach space and let $E : (B, \| \ . \ \|_1) \longrightarrow \mathbb{R}$ be a function of class $C^2$ whose differential $Df$ vanishes at $0$. The Taylor-Young formula is then written
 \begin{equation}\label{TY}
    E(u) = E(0) + D^2 \ E(0) \ . \ (u \centerdot u) + o(|| u
    ||_1^2)
 \end{equation}
 In particular, if the Hessian form $D^2 E(0)$ is coercive in the norm $\| \ . \ \|_1$, then the critical point  $0$ is a strict local  minimum of $E$. The fundamental difficulty in the study of critical forms is caused by  the appearance of a second norm $\| \ . \ \|_2$ lower that $\| \ . \ \|_1$ \ \ $(i.e \  \ \| \ . \ \|_2 \leq C \| \ . \ \|_1)$. The Hessian form is not in general, coercive for the norm $\| \ . \ \|_1$ but for the standard norm $\| \ . \ \|_2$. If these norms are not equivalent and this is the general rule, concluding that the minimum is strict is impossible, even locally for the strong norm. It is quite possible to give several example. But let us reproduce
a simple example of such a situation on the space $H_0^1(0, 1)$ that was presented in the thesis of \cite{Da2}. Let us   consider the functional $E$ defined by
 $$E(u) = \|u\|^2_{L^2(0, 1)} - \|u\|^4_{H^1_0(0, 1)}.$$
We can check that $E$ is twice differentiable on
$H_0^1(0, 1)$ and which moreover one has  in $0$ :
$$\left\{\begin{array}{l}
  E'(0) = 0 \\
  E''(0).(h, h) = 2\|h\|^2_{L^2(0, 1).}
\end{array}\right.$$
For each direction, we find that $0$ is a minimum
strictly local. indeed, for all nonzero $u_0 \in H_0^1(0, 1)$ and
for all $t \in \mathbb{R}$, we have
$$E(tu_0) =t^2\|u_0\|^2_{L^2(0, 1)} - t^4\|u_0\|^4_{H^1_0(0, 1))} > 0 \ if \ t^2 < \displaystyle\frac{\|u_0\|^2_{L^2(0, 1)}}{\|u_0\|^4_{H^1_0(0, 1)}}.$$
However, $0$ is not a local minimum even for the norm
$H_0^1$. Indeed, there is no $r > 0$ such that $$\|u\|_{H_0^1(0, 1)} < r \Longrightarrow E(u) > E(0) = 0 \ \ \textrm{i.e} \ \ \|u\|^2_{L^2(0, 1)} > \|u\|^4_{H_0^1(0, 1)},$$
since we can always build a sequence in $H_0^1(0, 1)$
such that
$$\left\{\begin{array}{l}
  \|u_n\|_{H_0^1(0, 1)} = r/2, \\
  \|u_n\|_{L^2(0, 1)} \longrightarrow 0 \ \textrm{quand} \ n \longrightarrow +\infty.
\end{array}\right.$$
To solve this problem, we will use the Taylor's formula  with  an  integral remainder, instead of $(\ref{TY})$ \  i.e
$$E(u) - E(0) = \displaystyle\int_0^1 \ (1 - t)  \ E''(tu)(u, u) \ dt.$$
This formula allows to express exactly the difference in energy
between a critical form  $\Omega_0$ and a neighboring form $\Omega$
via an integral term that we can carefully estimate thanks to the study of the variations of the Hessian.
\begin{theorem}
Let $f_0 : \mathbb{R} \times \mathbb{R}^N \longrightarrow
\mathbb{R}, \ (s, v) \longmapsto f_0(s, v)$ be a function of class
$\mathcal C^3$ and $f$ a function in $C^{0, \gamma}(\mathbb{R}^N,
\mathbb{R}), \gamma\in (0, 1)$. Let $L_0 = div(A \nabla.)$ be a strictly operator and uniformly elliptical with $A$ in $\mathcal C^2( \mathbb{R}^N, \
M_N(\mathbb{R}^N))$. Let $E$ be the defined shape functional on the class $\mathcal O $ of open class $\mathcal{C}^{2, \gamma}$ as
$$J(\Omega) = \displaystyle\int_\Omega \  f_0(u_\Omega, \  \nabla u_\Omega),$$
where $M_N(\mathbb{R}^N))$ stands for the space of square matrices of order $N$ and   $u_\Omega$ is the solution of the homogeneous Dirichlet problem
$$\left\{\begin{array}{l}
    L_0 u = f \ \ in \ \ \Omega, \\
    u =  0 \ \ on \ \ \partial \Omega.
  \end{array}\right.$$
  Let $\Omega_0\in \mathcal O, $ then, there exists a real $\eta_0 > 0$ and an increasing function $\omega : (0, \eta_0] \longrightarrow (0, +\infty)$ with $\displaystyle\lim_{r\searrow 0} \omega(r) = 0$, which depend only on $\Omega_0, \ L_0, \ f_0 $ and $ f,$ such that for all $\eta \in (0, \eta_0]$ and for all $\theta \in \mathcal{C}^{2, \alpha} ( \mathbb{R}^N, \ \mathbb{R}^N)$ verifying
  $$\|\theta - Id_{\mathbb{R}^N}\|_{2, \alpha} \leq \eta,$$ we have the following estimate valid for all $t$ in $[0, 1]$,
  \begin{equation}\label{l}
    \biggl|\displaystyle\frac{d^2}{dt^2} \ J(\Omega_t ) - \displaystyle\frac{d^2}{dt^2}_{|t=0} \ J(\Omega_t) \biggr| \leq \omega(\eta)\| < V , \vec{\nu} >  \|^2_{H^{1/2}(\partial
    \Omega_0)}
  \end{equation}
  where $\Omega_t=  \Phi_t(\Omega_0), t \in [0, 1]$ stands for the flow related to the vector field $V.$
\end{theorem}
  For the proof see (\cite{Da1}, Theorem $1$).

  
In the case where $\O_0$  is  a critical point  for  the functional $J,$ to show that  it is a strict local minimum,
 we have  
to   study  the positiveness of a quadratic form  which we are going to denote by $ Q$ . This
quadratic form is obtained by  computing  the second derivative of
$J$ with respect to the domain. So before going on, we need some hypothesis ;\\
 let us suppose that:
 \begin{description}
 \item{(i)\,\,-} $\O$ is a ${\mathcal C}^{2}-$ regular \,\, open domain.
 \item{(ii)\,\,-} $V(x;t) = \alpha(x)\vec\nu(x),\,\,\,\alpha\,\,\in\,\,H^{\frac{1}{2}}(\dpa \O),\,\,\forall\,\,t\,\,\,\in\,\,[0,\epsilon[$.
 \end{description}
 In \cite{DaPi}, (see also \cite{Da1}, \cite{Da2}), the authors showed that it is not sufficient
 to prove that the quadratic form is positive to claim that: a critical
 shape is a minimum.
 In fact most of the time people use the Taylor Young formula to study
 the positiveness of the quadratic form.\\
 For $t \in [0, \epsilon[$, $j(t):= J(\O_t)=J(\O)+ t dJ(\O, V)+ \frac{1}{2}t^2
 d^2J(\O,V,V)+o(t^2).$\\
 The quantity $o(t^2)$ is expressed with the norm of $\mathcal
 C^2$. It appears in the expression of $d^2J(\O,V,V)$ the norm
 of $H^{\frac{1}{2}}(\dpa \O)$. And these two norms are not equivalent.
 The quantity $o(t^2)$ is not smaller than $||V||_{H^{\frac{1}{2}}(\dpa
 \O)}$, see the example in \cite{DaPi}.
 Then such an argument does not insure that the critical point is
 a local strict minimum.\\
 
In our study, we shall see that  the main result in
 \cite{DaPi}  can be satisfied in a simple way thanks to the hessian obtained via the Sobolev metric $G^A$ in which the norm of $H^{1/2}(\partial \Omega)$ appears directly. And this overcomes the clasisical issue. In fact the study of the sign of $\displaystyle \int_{\partial \Omega} H d \sigma$ becomes the  alone  control to get information on the optimal domain and then on the optimal shape
\begin{propo}\label{pro2}
$  $  \\
Let \,\,$\O$\,\   be a critical point  for the functional $J,$ then
\begin{eqnarray*}
 Q(\alpha)  & = & d^{2}J (\O;V;V) \\
       & =  & -( N - 1) \disp\int_{\dpa \O} H \alpha^{2}d\sigma + k^2 \disp\int_{\O} |\na \Lambda|^{2} dx  \\
       & =  &  -(N - 1) k^2 \disp\int_{\dpa \O} H \alpha^{2}d\sigma +  k^2 \disp\int_{\dpa \O}  \alpha L \alpha d\sigma
\end{eqnarray*}
Where $\Lambda$ is solution of the following boundary value problem
\begin{eqnarray}
\left \{
\begin{array}{rclcl}
-\D \Lambda  & = & 0 &\mbox{in}& \O\\
 \Lambda & = &  \alpha & \mbox{on} &\dpa \O .
\end{array}
\right.
\end{eqnarray}
H is the mean  curvature of $\dpa \O$ \,\,and L is a pseudo
differential operator which is known as the Steklov-Poincar\'e or
capacity  or Dirichlet to Neumann(see e.g \cite{DL})
operator, defined by  $ L \alpha =\disp\frac{\dpa \Lambda}{\dpa \vec\nu}$.\\
\end{propo}
\begin{proof}
We use the definition of the derivative  with respect to the domain and we apply it to  $ dJ(\O,V).$\\
Then we get
\begin{eqnarray*}
 2 Q(\alpha) & = &2 d^{2} J(\O,V,V)\\
      & = &  \disp\int_{\O} ( \div((k^2 - |\na u|^{2}) V(x,0)))^{\prime} dx +
      \disp\int_{\O} \div( V(x,0) \div((k^2 - |\na u|^{2}) V(x,0))) dx \\
 2 Q(\alpha) & = & \left[\disp\int_{\dpa \O}  -2\na u \na u^{\prime} V(x,0).\vec\nu  +
 \div((k^2 - |\na u|^{2}) V(x,0)) V(x,0).\vec\nu \right] d\sigma .
\end{eqnarray*}

 Since $\O$ is solution of the quadradure surface problem then   $- \disp\frac{\dpa u}{\dpa \vec\nu} = k\,\,\,\,\mbox{on}\,\,\,\dpa \O.$\\
By assumption,   $\dpa \O$\,\,is of ${\mathcal C}^{2}$\,\,class and  since $u = 0\,\,\,\mbox{on}\,\,\,\dpa \O,$\\
\,\,we have : \\$\na u  = \disp\frac{\dpa u}{\dpa \vec\nu} \vec\nu = - k\vec\nu $.\,\,\,Hence
$$
2 Q(\alpha)  =  \left[\disp\int_{\dpa \O}  2 k \na u^{\prime}. \vec\nu  V(x,0). \vec\nu  +
 \div((k^2 - |\na u|^{2}) V(x,0) )V(x,0).\vec\nu \right] d\sigma .
$$
A classical calculus in shape optimization lead us to  get $ u^{\prime} = -
\disp\frac{\dpa u}{\dpa \vec\nu} V.\vec\nu \,\,\,\mbox{on}\,\,\,\dpa
\O.$\,\,\,
\\
Let us recall again that $- \disp\frac{\dpa u}{\dpa \vec\nu} = k\,\,\,\,\mbox{on}\,\,\,\dpa \O$ and  $V.\vec\nu=\alpha.$  Then, we have 
$ u^{\prime}  =   k\alpha\,\,\,\mbox{on}\,\,\,\dpa \O $\,\,and \,\,\,$$ \vec\nu \na u^{\prime} =
 \disp\frac{\dpa u^{\prime}}{\dpa \vec\nu} =  k \disp\frac{\dpa \alpha}{\dpa \vec\nu} =  k L\alpha, $$ where
$ L$ is a pseudo differential operator, defined by  $ L\alpha =\disp\frac{\dpa \Lambda}{\dpa \vec\nu}$
 and such that
\begin{eqnarray}
\left \{
\begin{array}{rclcl}
-\D \Lambda  & = & 0 &\mbox{in}& \O\\
 \Lambda & = &  \alpha & \mbox{on} &\dpa \O .
\end{array}
\right.
\end{eqnarray}
$\Lambda$ is the extension of $\alpha$ in $\O.$\\
Hence
$$ 2 Q(\alpha) = \disp\int_{\dpa \O} ( 2 k^2 \alpha L\alpha  - \div((|\na u|^{2} -k^2) \alpha.\vec\nu)\alpha) d\sigma$$
Let  us compute now  $\div((|\na u|^{2} -k^2) \alpha.\vec\nu)\,\,\,\,\mbox{on}\,\,\,\dpa \O.$  Since $|\na u| = k\,\,\,\,\mbox{on}\,\,\,\dpa \O,$ \,\,we have
$$  \div( (|\na u|^{2} -k^2) \alpha.\vec\nu) =  \alpha \na (|\na u|^{2} -k^2). \vec\nu =  \alpha \na (|\na u|^{2}).\vec\nu$$
 Since we have  supposed that $\Omega$ of class $\mathcal C^2,$ locally,  $\partial \Omega$ can be  described by a curve $\varphi$ such that $x_N = \varphi(x'), x' \in \mathbb{R}^{N-1}$ \ and  \  $D\varphi(x') = 0. \   D\varphi(x')$ is the Jacobian matrix of $\varphi.$\\
 Let us set $x_0 = (x', x_N)=
(x', \varphi(x')) \in \partial \Omega$  
then we have
$u(x_0) = 0.$\\ By differentiating with respect to $s_j$ for all $j\in \{1, \cdots, N-1\}, $ we have:
$$ \displaystyle\frac{\partial u(x_0)}{\partial s_j} +
\displaystyle\frac{\partial \varphi \ (x')}{\partial s_j} \
\displaystyle\frac{\partial u  \ (x_0)}{\partial \vec\nu} = 0 $$
 Since $ \displaystyle\frac{\partial \varphi(x')}{\partial s_j}
= 0,$  we get  $ \displaystyle\frac{\partial
u(x_0)}{\partial s_j} = 0.$\\ 
Starting from the following equality:
\begin{equation}\label{ }
\displaystyle\frac{\partial u(x_0)}{\partial s_j} +
\displaystyle\frac{\partial \varphi(x')}{\partial s_j} \ \
\displaystyle\frac{\partial u(x_0)}{\partial \vec\nu} = 0, 
\end{equation}
 and by differentiating it from $i\in \{1, \cdots, N-1\},$ we have:
$$\displaystyle\frac{\partial^2 u(x_0)}{\partial s_i \ \partial s_j} +
\displaystyle\frac{\partial \varphi(x')}{\partial s_i} \ \
\displaystyle\frac{\partial^2 u(x_0)}{\partial \vec\nu \ \partial s_j}  +
\displaystyle\frac{\partial^2\varphi (x')}{\partial s_i \ \partial
s_j}  \    \displaystyle\frac{\partial u(x_0)}{\partial \vec\nu}$$
$$+ \displaystyle\frac{\partial \varphi (x')}{\partial s_j } \ \
\displaystyle\frac{\partial^2 u(x_0)}{\partial s_i \ \partial \vec\nu} +
\displaystyle\frac{\partial \varphi(x')}{\partial s_j} \ \
\displaystyle\frac{\partial \varphi (x')}{\partial s_i}  \
\displaystyle\frac{\partial^2 u(x_0)}{\partial \vec\nu^2} = 0 
$$
 Note that $u(x_0) = 0 \ \mbox{and}\   \displaystyle\frac{\partial u (x_0)}{\partial s_j} = 0\, \,  \forall j\in \{1, \cdots, N-1\}.$  
and summing over the indices \ $i, j,$ we have  \\

\begin{equation}\label{ }
\displaystyle\sum_{j=1}^{N-1} \ \displaystyle\frac{\partial^2
u(x_0)}{\partial s_j^2} + (N - 1) \ H \
 \displaystyle\frac{\partial  u(x_0)}{\partial \vec\nu} = 0
 \end{equation}
 Since$\displaystyle\frac{\partial u(x_0)}{\partial s_i} = 0 \, \,  \forall i\in \{1, \cdots, N-1\}, $ we have also:\\
 $\begin{array}{lll}
    \nabla(|\nabla u|^2(x_0)).\vec\nu & = & \frac{\partial}{\partial \vec\nu}
    \biggl[\displaystyle\sum_{i=1}^{N-1} \ \biggl(\displaystyle\frac{\partial u(x_0)}{\partial  s_i }\biggr)^2
     +  \biggl(\displaystyle\frac{\partial u(x_0)}{\partial \vec\nu}\biggr)^2\biggr]  \    \\
     & = & 2  \ \displaystyle\frac{\partial u(x_0)}{\partial \vec\nu} \ \ \displaystyle\frac{\partial^2 u(x_0)}{\partial \vec\nu^2} \end{array} $\\
   In addition,  we can  remark that:\\
     $\displaystyle\frac{\partial^2 u(x_0)}{\partial \vec\nu^2} =
     - \displaystyle\sum_{i=1}^{N-1} \ \displaystyle\frac{\partial^2 u(x_0)}{\partial s_i^2} -f$ on $\partial \Omega$\\
    Therefore, we have 
    \begin{eqnarray*}  
      \nabla (|\nabla u|^2(x_0)).\vec\nu=  2  \ \displaystyle\frac{\partial u(x_0)}{\partial \vec\nu} \left(  -\sum_{i=1}^{N-1} \frac{\partial^2 u(x_0)}{\partial s_i^2} -f\right)\\
      =  2  \ \displaystyle\frac{\partial u(x_0)}{\partial \vec\nu} \left(  (N - 1) \ H \
 \displaystyle\frac{\partial  u(x_0)}{\partial \vec\nu} -f\right)
    \end{eqnarray*}  
When the support of the function $f$ is in $\Omega,$ then $f=0$ on $\partial \Omega.$\\
Finally we have:
\begin{eqnarray*}
   2 Q(\alpha) = \disp\int_{\dpa \O}  2 k^2 \alpha L\alpha  - 2   (N - 1)  H \alpha^2    \ \displaystyle\left(\frac{\partial u(x_0)}{\partial \vec\nu}\right)^2  d\sigma\\
    =  \disp\int_{\dpa \O}  2 k^2 \alpha L\alpha  - 2  k^2  (N - 1)  H \alpha^2     d\sigma
    \end{eqnarray*}
 And by the Green's formula we get
$$  \disp\int_{\dpa \O} \alpha L \alpha  d\sigma = \disp\int_{\O} |\na \Lambda|^{2} dx.$$ \cqfd
\end{proof}
\section{Positiveness of the quadratic form in the infinite Riemannian point of view}
\begin{definition}
Let $J: \Omega\to\mathbb{R}$ be an functional. One defines the hessian Riemannian shape as follows:
$$
Hess J(\Omega)[V]:=\nabla_{V}grad J 
$$
where $\nabla_{V}$ denotes the derivative following the vector field $V$.
\end{definition}
\begin{theorem}\label{punk0}
The hessian Riemannian shape defined by the Riemannian metric $G^{A}$ verifies the following condition:
$$
G^{A}(Hess J(\Omega)[V],W)=d(dJ(\Omega)[W])[V]-dJ(\Omega)[\nabla_{V}W].
$$
\end{theorem}
\begin{proof}
Our purpose is to show that 
$$
G^{A}(Hess J(\Omega)[V],W)=d(dJ(\Omega)[W])[V]-dJ(\Omega)[\nabla_{V}W]
$$
So  let us  use the compatibility of the metric $G^{A}$ with the Levi-Civita  connection. We have
\begin{eqnarray}
V.G^{A}(grad J,W)&=&G^{A}(grad J,\nabla_{V}{W})+G^{A}(\nabla_{V}{grad J},W)\nonumber\\
G^{A}(\nabla_{V}{grad J},W)&=&V.G^{A}(grad J, W)-G^{A}(grad J,\nabla_{V}W)\nonumber.
\end{eqnarray}
Since $G^{A}(Hess J(\Omega)[V], W)=G^{A}(\nabla_{V}grad J , W),$ we have
\begin{eqnarray}
G^{A}(Hess J(\Omega)[V],W)&=&V.G^{A}(grad J,W)-G^{A}(grad J,\nabla_{V}W)\nonumber\\
G^{A}(Hess J(\Omega)[V],W) &=&V.(WJ)-(\nabla_{V}W).J\nonumber\\
 G^{A}(Hess J(\Omega)[V],W) &=&d(dJ(\Omega)[W])[V]-dJ(\Omega)[\nabla_{V}W]\nonumber
\end{eqnarray}
where  $V,W\in\mathcal{C}^{\infty}(\mathbb{R}^{2},\mathbb{R}^{2})$ are vector fields orthogonal to the boundary $\partial\Omega$ and $d(dJ(\Omega)[W])[V]$ define the standard Hessian shape.
\end{proof}
\begin{remark}
In our quadrature surface case, for $W= m\vec{\nu}$ and $V= h \vec{\nu},$ we have:
 $$dJ(\Omega)[V]=\int_{\partial\Omega}\left(k^{2}-\left(\frac{\partial u_{\Omega}}{\partial\vec\nu}\right)^{2}\right)\alpha d\sigma $$  and then,
$$ d \left(dJ(\Omega)[W]\right)[V]=d \left(\int_{\partial\Omega}\left(k^{2}-\left(\frac{\partial u_{\Omega}}{\partial\vec\nu}\right)^{2} \right) m d\sigma \right)[V].$$
Setting $$\psi:= k^{2}-\left(\frac{\partial u_{\Omega}}{\partial\vec\nu}\right)^{2},$$ then $$\psi_{t}:= k^{2}-\left(\frac{\partial u_{\Omega_{t}}}{\partial\vec\nu}\right)^{2},$$ we have
\begin{eqnarray} 
d\left(dJ(\Omega)[W]\right)[V]&=& d\left(\int_{\partial\Omega_t}\psi_t m d\sigma\right)[h].\nonumber
\end{eqnarray}
This is never but:
\begin{eqnarray}
d\left(dJ(\Omega)[W]\right)[V]&=&\int_{\partial\Omega}\frac{\partial \psi_{t}}{\partial t}_{|t=0}md\sigma+\int_{\partial\Omega}\frac{\partial(\psi m)}{\partial\vec{\nu}}h d\sigma+\int_{\partial\Omega}K_{c}\psi m h d\sigma\nonumber\\
&=&\int_{\partial \Omega}\left[\frac{\partial \psi_{t}}{\partial t}_{|t=0}m+\left(\frac{\partial\psi}{\partial\vec{\nu}}+K_{c}\psi\right)m h+\psi\frac{\partial m}{\partial\vec{\nu}}h\right]d\sigma\nonumber
\end{eqnarray}
Let us compute 
\begin{eqnarray}
\frac{\partial \psi_{t}}{\partial t}_{|t=0}m&=&\frac{\partial\left[k^{2}-\left(\frac{\partial u_{\Omega_{t}}}{\partial\vec\nu_t}\right)^{2}\right]}{\partial t}_{|t=0}m\nonumber\\
\frac{\partial \psi_{t}}{\partial t}_{|t=0}m&=&-2m\left[\frac{\partial (\nabla u_{\Omega_t}).\vec\nu}{\partial t}_{|t=0}+D^{2}u_{\Omega}V.\vec\nu+\nabla u_{\Omega}.\left(\frac{\partial\vec\nu_{t}}{\partial t}_{|t=0}+D\vec\nu V \right)\right]\nonumber\\
\frac{\partial \psi_{t}}{\partial t}_{|t=0}m&=&-2m\left[\nabla u_{\Omega}^{\prime}.\vec\nu+D^{2}u_{\Omega}V.\vec\nu+\nabla u_{\Omega}.\left(\frac{\partial\vec\nu_{t}}{\partial t}_{|t=0}+D\vec\nu V \right)\right]\nonumber\\
\end{eqnarray} 
where $D^2 u_{\Omega}$ is the hessian matrix and $D\vec\nu$ the jacobian matrix of $\vec\nu$ .\\Let us calculate now the following expression: $\frac{\partial\vec\nu_{t}}{\partial t}_{|t=0}.$\\We have $$\frac{\partial\vec\nu_{t}}{\partial t}_{|t=0}=-\nabla_{\Gamma}(V.\vec\nu)-(D\vec\nu_{0}.\vec\nu)V.\vec\nu \ \ \mbox{on}\ \ \Gamma$$ where $\nabla_{\Gamma}$ is the tangential gradient, $\Gamma=\partial\Omega$ and $\vec\nu_{0}=\vec\nu$  then $$\frac{\partial\vec\nu_{t}}{\partial t}_{|t=0}=-\nabla_{\Gamma}(V.\vec\nu)-(D\vec\nu.\vec\nu)V.\vec\nu \ \ \mbox{on}\ \ \Gamma.$$  Since $D\vec\nu.\vec\nu\equiv 0, $ then $$\frac{\partial\vec\nu_{t}}{\partial t}_{|t=0}=-\nabla_{\Gamma}(V.\vec\nu) \ \ \mbox{on}\ \ \Gamma.$$ So $$\frac{\partial \psi_{t}}{\partial t}_{|t=0}m=-2m\left[\nabla u_{\Omega}^{\prime}.\vec\nu+D^{2}u_{\Omega}V.\vec\nu+\nabla u_{\Omega}.\left(-\nabla_{\Gamma}(V.\vec\nu)+D\vec\nu.V \right)\right]$$
And finally, we get
\begin{eqnarray}
d\left(dJ(\Omega)[W]\right)[V]&=&\int_{\partial \Omega}\left[-2m\big(\nabla u_{\Omega}^{\prime}.\vec\nu+D^{2}u_{\Omega}V.\vec\nu+\nabla u_{\Omega}.\left(-\nabla_{\Gamma}(V.\vec\nu)+D\vec\nu V \right)\big)+\left(\frac{\partial\psi}{\partial\vec{\nu}}+K_{c}\psi\right)m h+\psi\frac{\partial m}{\partial\vec{\nu}}h\right]d\sigma\nonumber\\
d\left(dJ(\Omega)[W]\right)[V]&=&\int_{\partial \Omega}\left[-2\big<W,\vec{\nu}\big>\big(\nabla u_{\Omega}^{\prime}.\vec\nu\right.+D^{2}u_{\Omega}V.\vec\nu+\nabla u_{\Omega}.\left(-\nabla_{\Gamma}(V.\vec\nu)+D\vec\nu V \right)\big)\nonumber\\
&+&\left. \left(\frac{\partial\psi}{\partial\vec{\nu}}+K_{c}\psi\right)\big<W,\vec{\nu}\big> \big<V,\vec{\nu}\big>+\psi\big<D_{V}W,\vec{\nu}\big>\right]d\sigma\nonumber.
\end{eqnarray}
\end{remark}
On the one hand, having the following  Riemannian hessian formula
\begin{eqnarray}
G^{A}\left(Hess J(\Omega)[V], W \right)&=&d\left(dJ(\Omega)[W]\right)[V]-dJ(\Omega)[\nabla_{V}W],\nonumber
\end{eqnarray}
it is possible to  bring additional details on its computation.
\begin{proposition}
We have:
\begin{eqnarray}\label{punk1}
G^{A}(Hess J(\Omega)[V],W)&=&\int_{\partial \Omega}\left[-2\big<W,\vec{\nu}\big>\big(\nabla u_{\Omega}^{\prime}.\vec\nu+D^{2}u_{\Omega}V.\vec\nu+\nabla u_{\Omega}.\left(-\nabla_{\Gamma}(V.\vec\nu)+D\vec\nu V \right)\big)\right]d\sigma\nonumber\\
&+&\int_{\partial\Omega}\left[\frac{\partial}{\partial\vec{\nu}}\left(k^{2}-\left(\frac{\partial u_{\Omega}}{\partial\vec{\nu}}\right)^{2}\right)+K_{c}\left(k^{2}-\left(\frac{\partial u_{\Omega}}{\partial\vec{\nu}}\right)^{2}\right)\right.\nonumber\\
&-&\left.\frac{3AK_{c}^{3}+K_{c}}{1+AK_{c}^{2}}\left(k^{2}-\left(\frac{\partial u_{\Omega}}{\partial\vec{\nu}}\right)^{2}\right)\right]\big<V,\vec{\nu}\big>\big<W,\vec{\nu}\big>d\sigma
\end{eqnarray}
\end{proposition}
\begin{proof}
\begin{eqnarray}
G^{A}\left(Hess J(\Omega)[V], W \right)&=&\int_{\partial \Omega}\left[-2\big<W,\vec{\nu}\big>\big(\nabla u_{\Omega}^{\prime}.\vec\nu\right.+D^{2}u_{\Omega}V.\vec\nu+\nabla u_{\Omega}.\left(-\nabla_{\Gamma}(V.\vec\nu)+D\vec\nu V \right)\big)\nonumber\\
&+&\left. \left(\frac{\partial\psi}{\partial\vec{\nu}}+K_{c}\psi\right)\big<W,\vec{\nu}\big> \big<V,\vec{\nu}\big>+\psi\big<D_{V}W,\vec{\nu}\big>\right]d\sigma\nonumber\\
        &-&\int_{\partial\Omega}\psi\big<\nabla_{V}W,\vec{\nu}\big>d\sigma\nonumber\\
&=&\int_{\partial \Omega}\left[-2\big<W,\vec{\nu}\big>\big(\nabla u_{\Omega}^{\prime}.\vec\nu\right.+D^{2}u_{\Omega}V.\vec\nu+\nabla u_{\Omega}.\left(-\nabla_{\Gamma}(V.\vec\nu)+D\vec\nu V \right)\big)\nonumber\\
&+&\left. \left(\frac{\partial\psi}{\partial\vec{\nu}}+K_{c}\psi\right)\big<W,\vec{\nu}\big> \big<V,\vec{\nu}\big>+\psi\big<D_{V}W,\vec{\nu}\big>\right]d\sigma\nonumber\\
&-&\int_{\partial\Omega}\psi\left[\big<D_{V}W,\vec{\nu}\big>+\left(\frac{3AK^{3}_{c}+K_{c}}{1+AK_{c}^{2}}\right)\big<V,\vec{\nu}\big>\big<W,\vec{\nu}\big>\right]d\sigma\nonumber\\
&=&\int_{\partial \Omega}\left[-2\big<W,\vec{\nu}\big>\big(\nabla u_{\Omega}^{\prime}.\vec\nu+D^{2}u_{\Omega}V.\vec\nu+\nabla u_{\Omega}.\left(-\nabla_{\Gamma}(V.\vec\nu)+D\vec\nu V \right)\big)\right]d\sigma\nonumber\\
&+&\int_{\partial\Omega}\left[\frac{\partial\psi}{\partial\vec{\nu}}+K_{c}\psi-\psi K_c \left(\frac{3AK^{2}_{c}+1}{1+AK_{c}^{2}}\right)\right]\big<V,\vec{\nu}\big>\big<W,\vec{\nu}\big>d\sigma\nonumber
\end{eqnarray}
Replacing $\psi$ by its expression, we have:
\begin{eqnarray}\label{punk1}
G^{A}(Hess J(\Omega)[V],W)&=&\int_{\partial \Omega}\left[-2\big<W,\vec{\nu}\big>\big(\nabla u_{\Omega}^{\prime}.\vec\nu+D^{2}u_{\Omega}V.\vec\nu+\nabla u_{\Omega}.\left(-\nabla_{\Gamma}(V.\vec\nu)+D\vec\nu V \right)\big)\right]d\sigma\nonumber\\
&+&\int_{\partial\Omega}\left[\frac{\partial}{\partial\vec{\nu}}\left(k^{2}-\left(\frac{\partial u_{\Omega}}{\partial\vec{\nu}}\right)^{2}\right)+K_{c}\left(k^{2}-\left(\frac{\partial u_{\Omega}}{\partial\vec{\nu}}\right)^{2}\right)\right.\nonumber\\
&-&\left.\frac{3AK_{c}^{3}+K_{c}}{1+AK_{c}^{2}}\left(k^{2}-\left(\frac{\partial u_{\Omega}}{\partial\vec{\nu}}\right)^{2}\right)\right]\big<V,\vec{\nu}\big>\big<W,\vec{\nu}\big>d\sigma
\end{eqnarray}
\end{proof}

On the other hand, 
 let us  compute $G^{A}(Hess J(\Omega)[V],W)$ by using directly the Sobolev-type metric $G^A$.  Then we have the following proposition.
 \begin{proposition}
 \begin{eqnarray}\label{punk2}
G^{A}(Hess J(\Omega)[V],W)=
\int_{\partial\Omega}\left[\frac{\partial}{\partial\vec\nu}\left(k^{2}-\left(\frac{\partial u_{\Omega}}{\partial\vec{\nu}}\right)^{2}\right)+ K_{c}\left(k^{2}-\left(\frac{\partial u_{\Omega}}{\partial\vec{\nu}}\right)^{2}\right)\right]\big<V,\vec{\nu}\big>\big<W,\vec{\nu}\big> d\sigma.
\end{eqnarray}
 \end{proposition}
 \begin{proof}
\begin{eqnarray}
G^{A}(Hess J(\Omega)[V],W)&=&\int_{\partial\Omega}\left(1+AK^{2}_{c}\right)Hess J(\Omega)[V]W\nonumber\\
&=&\int_{\partial\Omega}\left(1+AK^{2}_{c}\right)\nabla_{V}gradJ(\Omega)W\nonumber\\
 &=&\int_{\partial\Omega}\left(1+AK^{2}_{c}\right)\nabla_{h}gradJ(\Omega)m\nonumber
\end{eqnarray}
 Since  $gradJ(\Omega)=\frac{1}{1+AK^{2}_{c}}\psi$,we have
\begin{eqnarray}
\nabla_{h}gradJ(\Omega)&=&\frac{\partial}{\partial\vec\nu}\left(gradJ(\Omega)\right)\alpha+\left(\frac{3AK^{3}_{c}+K_{c}}{1+AK_{c}^{2}}\right)gradJ(\Omega)\alpha\nonumber\\
&=&\frac{\partial}{\partial\vec\nu}\left(\frac{1}{1+AK^{2}_{c}}\psi\right)\alpha+\frac{1}{1+AK^{2}_{c}}\psi\left(\frac{3AK^{3}_{c}+K_{c}}{1+AK_{c}^{2}}\right)\alpha\nonumber\\
&=&\frac{\partial}{\partial\vec\nu}\left[(1+AK^{2}_{c})^{-1}\right]\psi\alpha+\frac{\partial\psi}{\partial\vec\nu}\left(\frac{1}{1+AK^{2}_{c}}\right)\alpha+\frac{1}{1+AK^{2}_{c}}\psi\left(\frac{3AK^{3}_{c}+K_{c}}{1+AK_{c}^{2}}\right)\alpha\nonumber\\
&=&-2AK_{c}\frac{\partial K_{c}}{\partial\vec\nu}\left(1+AK^{2}_{c}\right)^{-2}\psi\alpha+\frac{\partial\psi}{\partial\vec\nu}\left(\frac{1}{1+AK^{2}_{c}}\right)\alpha\nonumber\\&+&\frac{1}{1+AK^{2}_{c}}\psi\left(\frac{3AK^{3}_{c}+K_{c}}{1+AK_{c}^{2}}\right)\alpha\nonumber
\end{eqnarray}
Note that $\frac{\partial K_{c}}{\partial\vec\nu}= K_c^2,$ what implies that:
\begin{eqnarray}
\nabla_{h}gradJ(\Omega)
&=&\frac{-2AK^{3}_{c}}{\left(1+AK^{2}_{c}\right)^{2}}\psi\alpha+\frac{\partial\psi}{\partial\vec\nu}\left(\frac{1}{1+AK^{2}_{c}}\right)\alpha+\frac{1}{1+AK^{2}_{c}}\psi\left(\frac{3AK^{3}_{c}+K_{c}}{1+AK_{c}^{2}}\right)\alpha\nonumber.
\end{eqnarray}
 Then, coming back to our hessian computation, we have:
\begin{eqnarray}
 G^{A}(Hess J(\Omega)[V],W)&=&\int_{\partial\Omega}\left(1+AK^{2}_{c}\right)\left[\frac{-2AK^{3}_{c}}{\left(1+AK^{2}_{c}\right)^{2}}\psi\alpha+\frac{\partial\psi}{\partial\vec\nu}\left(\frac{1}{1+AK^{2}_{c}}\right)\alpha\right.\nonumber\\&+&\left.\frac{1}{1+AK^{2}_{c}}\psi\left(\frac{3AK^{3}_{c}+K_{c}}{1+AK_{c}^{2}}\right)\alpha\right]\beta d\sigma\nonumber\\
 &=&\int_{\partial\Omega}\left[\frac{-2AK^{3}_{c}}{1+AK^{2}_{c}}\psi\alpha+\frac{\partial\psi}{\partial\vec\nu}\alpha+\psi\left(\frac{3AK^{3}_{c}+K_{c}}{1+AK_{c}^{2}}\right)\alpha\right]\beta d\sigma\nonumber\\
 &=&\int_{\partial\Omega}\left[\frac{\partial\psi}{\partial\vec\nu}+\psi\left(\frac{AK^{3}_{c}+K_{c}}{1+AK_{c}^{2}}\right)\right]\alpha\beta d\sigma\nonumber\\
&=&\int_{\partial\Omega}\left[\frac{\partial\psi}{\partial\vec\nu}+\psi K_{c}\left(\frac{1+AK^{2}_{c}}{1+AK_{c}^{2}}\right)\right]\alpha\beta d\sigma\nonumber
\end{eqnarray}
Replacing $\psi$ by its expression, we have:
\begin{eqnarray}\label{punk2}
G^{A}(Hess J(\Omega)[V],W)=
\int_{\partial\Omega}\left[\frac{\partial}{\partial\vec\nu}\left(k^{2}-\left(\frac{\partial u_{\Omega}}{\partial\vec{\nu}}\right)^{2}\right)+ K_{c}\left(k^{2}-\left(\frac{\partial u_{\Omega}}{\partial\vec{\nu}}\right)^{2}\right)\right]\big<V,\vec{\nu}\big>\big<W,\vec{\nu}\big> d\sigma
\end{eqnarray}
\end{proof}
\begin{remark}
  Let us note first  that there is a symmetry relation with respect to the hessian which is in the case of our considered Riemannian structure a self adjoint operator with respect to the metric $G^A.$\\
And the second fact is that it is important to underline that the formulas $(\ref{punk1})$ obtained from the formula in {\bf Theorem $\ref{punk0}$} and  $(\ref{punk2})$ computed by a direct method with  the metric $G^A$ in  two different ways have to give the same expression even  if $\Omega$ is not a critical point.  And then  from these computations, one deduces that 
 \begin{eqnarray*}
 \int_{\partial \Omega}\left[-2\big<W,\vec{\nu}\big>\big(\nabla u_{\Omega}^{\prime}.\vec\nu+D^{2}u_{\Omega}V.\vec\nu+\nabla u_{\Omega}.\left(-\nabla_{\Gamma}(V.\vec\nu)+D\vec\nu V \right)\big)\right]d\sigma\\
= \int_{\partial \Omega}\frac{3AK_{c}^{3}+K_{c}}{1+AK_{c}^{2}}\left(k^{2}-\left(\frac{\partial u_{\Omega}}{\partial\vec{\nu}}\right)^{2}\right)\big<V,\vec{\nu}\big>\big<W,\vec{\nu}\big>d\sigma%
 \end{eqnarray*}
\end{remark}
\begin{remark}
In this remark, we 
 compute $G^{A}(V, Hess J(\Omega)[W])$  to show the symmetry relation with respect to the hessian with the computation of the direct method with  the metric $G^A.$
\begin{eqnarray}
G^{A}(V, Hess J(\Omega)[W])&=&\int_{\partial\Omega}\left(1+AK^{2}_{c}\right)Hess J(\Omega)[W]V\nonumber\\
&=&\int_{\partial\Omega}\left(1+AK^{2}_{c}\right)\nabla_{W}gradJ(\Omega)V\nonumber\\
 &=&\int_{\partial\Omega}\left(1+AK^{2}_{c}\right)\nabla_{m}gradJ(\Omega)h\nonumber
\end{eqnarray}
where $V= h= \alpha \vec{\nu}$ and $W= m= \beta \vec{\nu}.$
Since  $gradJ(\Omega)=\frac{1}{1+AK^{2}_{c}}\psi$, we have:
\begin{eqnarray}
\nabla_{m}gradJ(\Omega)&=&\frac{\partial}{\partial\vec\nu}\left(gradJ(\Omega)\right)\beta+\left(\frac{3Ak^{3}_{c}+K_{c}}{1+AK_{c}^{2}}\right)gradJ(\Omega)\beta\nonumber\\
&=&\frac{\partial}{\partial\vec\nu}\left(\frac{1}{1+AK^{2}_{c}}\psi\right)\beta+\frac{1}{1+AK^{2}_{c}}\psi\left(\frac{3Ak^{3}_{c}+K_{c}}{1+AK_{c}^{2}}\right)\beta\nonumber
\end{eqnarray}
As previously, by the same computations, we get:
\begin{eqnarray}
\nabla_{m}gradJ(\Omega)&=&\frac{-2AK^{3}_{c}}{\left(1+AK^{2}_{c}\right)^{2}}\psi\beta+\frac{\partial\psi}{\partial\vec\nu}\left(\frac{1}{1+AK^{2}_{c}}\right)\beta+\frac{1}{1+AK^{2}_{c}}\psi\left(\frac{3Ak^{3}_{c}+K_{c}}{1+AK_{c}^{2}}\right)\beta\nonumber.
\end{eqnarray}
And finally, we have: 
\begin{eqnarray}
 G^{A}(Hess J(\Omega)[W],V)&=&\int_{\partial\Omega}\left(1+AK^{2}_{c}\right)\left[\frac{-2AK^{3}_{c}}{\left(1+AK^{2}_{c}\right)^{2}}\psi\beta+\frac{\partial\psi}{\partial\vec\nu}\left(\frac{1}{1+AK^{2}_{c}}\right)\beta\right.\nonumber\\&+&\left.\frac{1}{1+AK^{2}_{c}}\psi\left(\frac{3AK^{3}_{c}+K_{c}}{1+AK_{c}^{2}}\right)\beta\right]\alpha d\sigma\nonumber\\
&=&\int_{\partial\Omega}\left[\frac{\partial}{\partial\vec\nu}\left(k^{2}-\left(\frac{\partial u_{\Omega}}{\partial\vec{\nu}}\right)^{2}\right)+ K_{c}\left(k^{2}-\left(\frac{\partial u_{\Omega}}{\partial\vec{\nu}}\right)^{2}\right)\right]\big<V,\vec{\nu}\big>\big<W,\vec{\nu}\big> d\sigma.\nonumber
\end{eqnarray}
\end{remark}
Let us have a look at on the two formulas of the second derivation when $V= W= \alpha \vec{\nu}.$\\
On the one hand, by Proposition $\ref{pro2},$ we get:
\begin{eqnarray*}
 Q(\alpha)  & = & d^{2}J (\O;V;V) \\
       & =  & -( N - 1) \disp\int_{\dpa \O} H \alpha^{2}d\sigma + k^2 \disp\int_{\O} |\na \Lambda|^{2} dx  \\
       & =  &  -(N - 1) k^2 \disp\int_{\dpa \O} H \alpha^{2}d\sigma +  k^2 \disp\int_{\dpa \O}  \alpha L \alpha d\sigma.
\end{eqnarray*}
On the other hand by Theorem $\ref{punk0},$ we have:
\begin{eqnarray*}
G^{A}\left(Hess J(\Omega)[V], W \right)&=&d\left(dJ(\Omega)[W]\right)[V]-dJ(\Omega)[\nabla_{V}W]\nonumber
\end{eqnarray*}
Then for $V=W$ we derive:
\begin{eqnarray*}
d\left(dJ(\Omega)[V]\right)[V]= d^{2}J (\O;V;V)=  G^{A}\left(Hess J(\Omega)[V], V \right)+dJ(\Omega)[\nabla_{V}V]\nonumber
\end{eqnarray*}
\begin{itemize}
\item If the quadrature surface problem has a solution $\Omega,$  then $d\left(dJ(\Omega)[V]\right)[V]=  G^{A}\left(Hess J(\Omega)[V], V \right).$
\item In previous works, the second author studied the stability and positiveness of the quadratic form, see \cite{sec} for more details. He established  a similar proposition   as  Proposition $\ref{pro2}$ and   gave  necessary and sufficient qualitative properties  in the theoretical point of view. 
\end{itemize}
The one obtained involves the study of a   generalized   spectral  Steklov problem that is reminded in the following corollary.
\begin{corollary}
Let us consider the  following   generalized   spectral  Steklov problem:
\begin{eqnarray*}
\Delta \phi_n &= &0 \; \;  \mbox{in}  \; \; \Omega\backslash K\\
\phi_n &=& 0  \; \; \mbox{on} \; \;  \partial K\\
(L+ (N-1) H I) \phi_n&=& (\frac{1}{\mu_n}- \|H^{-}\|_{\infty}) \phi_n \; \;   \mbox{on}  \; \;  \partial \Omega,
\end{eqnarray*}
where $I$ is the identity map, $H$ is the mean curvature of $\Omega,$  $K$ is a compact regular enough (let us say $\mathcal C ^2$) subset of $\Omega, H^-= max\{-H, 0\}$ and  $\mu_n$ is a  decreasing sequence of eigenvalues depending also on $H$  which goes to 0.
And one must have the sign of the first  eigenvalue $\displaystyle \lambda_0:= \frac{1}{\mu_0}- \|H^{-}\|_{\infty}= \inf\{(N-1)\int H v^2 d\sigma + \int_{\Omega\backslash K} |\nabla \Lambda |^2 dx, v \in H^{1/2}(\partial \Omega), \int_{\partial \Omega} v^2 d\sigma= 1\}, $ 
where 
\begin{eqnarray*}
\Delta \Lambda &= &0 \; \;  \mbox{in}  \; \; \Omega\backslash K\\
\Lambda &=& 0  \; \; \mbox{on} \; \;  \partial K\\
\frac{\partial \Lambda}{ \partial \vec{\nu}}&=& v \; \;   \mbox{on}  \; \;  \partial \Omega.
\end{eqnarray*}
 And  the minimum is  reached for $\phi_0$ satisfying
 \begin{eqnarray*}
\Delta \phi_0 &= &0 \; \;  \mbox{in}  \; \; \Omega\backslash K\\
\phi_0 &=& 0  \; \; \mbox{on} \; \;  \partial K\\
(L+ (N-1) H I) \phi_0&=& \lambda_0 \phi_0 \; \;   \mbox{on}  \; \;  \partial \Omega.
\end{eqnarray*}
\end{corollary}
From our work we can deduce the following conclusions as a corollary.
\begin{corollary}
\begin{itemize}
\item What is obtained with the Riemannian hessian formula is easier to derive simple control for the characterization of the optimal shape in a number of ways.
\item In the case of minimum, $G^{A}\left(Hess J(\Omega)[V], V \right)\geq 0.$ And this inequality is equivalent to $\displaystyle \int_{\partial\Omega}\left[\frac{\partial}{\partial\vec\nu}\left(k^{2}-\left(\frac{\partial u_{\Omega}}{\partial\vec{\nu}}\right)^{2}\right)\right]\alpha^2 d\sigma \geq 0 \, \, \forall \alpha \in \mathcal C^{\infty} (\mathbb{R}^2,  \mathbb{R})\cap H^{1/2}(\partial \Omega) .$\\
This is reduced to $ \displaystyle \int_{\partial\Omega}\frac{\partial}{\partial\vec\nu}\left(k^{2}-\left(\frac{\partial u_{\Omega}}{\partial\vec{\nu}}\right)^{2}\right)d\sigma\geq 0.$\\
One can deduce also another control, since \begin{eqnarray*}
\displaystyle \int_{\partial\Omega}\left[\frac{\partial}{\partial\vec\nu}\left(k^{2}-\left(\frac{\partial u_{\Omega}}{\partial\vec{\nu}}\right)^{2}\right)\right]\alpha^2 d\sigma= -2 k^2 (N- 1) \disp\int_{\dpa \O}  H \alpha^2     d\sigma= -2 k^2 (N- 1) \disp\int_{\dpa \O}  K_c \alpha^2     d\sigma .
\end{eqnarray*}
And knowing that  $\alpha \in  \mathcal C^{\infty} (\mathbb{R}^2,  \mathbb{R})\cap H^{1/2}(\partial \Omega),$ the control  becomes $ \disp\int_{\dpa \O}  K_c     d\sigma \leq 0.
$ 
Before proceeding further, let us underline that in two dimension $H=K_c.$\\
And from this, we   have  key information to set up algorithm in order to get a good  approximation of the optimal shape.
\item Now, when $\Omega$ is only a critical point, to get a strict local  minimum, we need   the following  sufficient condition:
 \begin{eqnarray*}
\displaystyle \int_{\partial\Omega}\left[\frac{\partial}{\partial\vec\nu}\left(k^{2}-\left(\frac{\partial u_{\Omega}}{\partial\vec{\nu}}\right)^{2}\right)\right]\alpha^2 d\sigma= -2 k^2 (N- 1) \disp\int_{\dpa \O}  K_c \alpha^2     d\sigma\geq C_0 \|\alpha\|^2, C_0 >0.
\end{eqnarray*}
One can say also that there is $x_0 \in \partial \Omega, -2 k^2 (N- 1) K_c (x_0) \disp\int_{\dpa \O}   \alpha^2     d\sigma\geq C_0 \|\alpha\|^2.$
And  if $K_c (x_0)< 0,$ then $ \Omega$ is a  strict local minimum.
\end{itemize}
\end{corollary}

\end{document}